\documentclass[11pt]{article}

\usepackage{fullpage}
\usepackage{amsmath}
\usepackage{amssymb}
\usepackage{amsthm}
\usepackage{fancyhdr}
\usepackage{graphicx,caption,subcaption}
\usepackage{algorithm}
\usepackage{algpseudocode}

\usepackage[usenames]{color}

\newtheorem{theorem}{Theorem}[section]
\newtheorem{lemma}[theorem]{Lemma}
\newtheorem{corollary}[theorem]{Corollary}

\newcommand{\appsection}[1]{\let\oldthesection\thesection
  \renewcommand{\thesection}{Appendix \oldthesection}
  \section{#1}\let\thesection\oldthesection}

\def\spn{\operatorname{span}}

 \def\eqd{\,{\buildrel d \over =}\,} 
  \def\leqd{\,{\buildrel d \over \leq}\,} 
 \def\geqd{\,{\buildrel d \over \geq}\,} 
 \def\rightarrowas{\,{\buildrel {a.s.} \over \rightarrow}\,}

  \def\E{{\Bbb E}}
\def\P{\Bbb P}
 \def\R{{\Bbb R}}

\definecolor{ejc}{RGB}{50,50,200}

\definecolor{ejcrw}{RGB}{255,0,0}

\definecolor{rw}{RGB}{50,200,50}

\numberwithin{equation}{section}

\title{Randomized Algorithms for Low-Rank Matrix Factorizations: Sharp
  Performance Bounds} 

\author{Rafi Witten\thanks{ Bit Body, Inc., Cambridge, MA} \, and
  Emmanuel Cand\`es\thanks{Departments of Mathematics and of
    Statistics, Stanford University, Stanford CA}} \date{August 2013}

\begin{document}
\maketitle

%\tableofcontents

\begin{abstract}The development of randomized algorithms for numerical
  linear algebra, e.g.~for computing approximate QR and SVD
  factorizations, has recently become an intense area of
  research. This paper studies one of the most frequently discussed
  algorithms in the literature for dimensionality
  reduction---specifically for approximating an input matrix with a
  low-rank element.  We introduce a novel and rather intuitive
  analysis of the algorithm in \cite{Rokhlin2008}, which allows us to
  derive sharp estimates and give new insights about its
  performance. This analysis yields theoretical guarantees about the
  approximation error and at the same time, ultimate limits of
  performance (lower bounds) showing that our upper bounds are
  tight. Numerical experiments complement our study and show the
  tightness of our predictions compared with empirical observations.
\end{abstract}

\section{Introduction}

Almost any method one can think of in data analysis and scientific
computing relies on matrix algorithms. In the era of `big data', we
must now routinely deal with matrices of enormous sizes and reliable
algorithmic solutions for computing solutions to least-squares
problems, for computing approximate QR and SVD factorizations and
other such fundamental decompositions are urgently
needed. Fortunately, the development of randomized algorithms for
numerical linear algebra has seen a new surge in recent years and we
begin to see computational tools of a probabilistic nature with the
potential to address some of the great challenges posed by big
data. In this paper, we study one of the most frequently discussed
algorithms in the literature for dimensionality reduction, and provide
a novel analysis which gives sharp performance bounds.

\subsection{Approximate low-rank matrix factorization}

We are concerned with the fundamental problem of approximately
factorizing an arbitrary $m \times n$ matrix $A$ as 
\begin{equation}
  \label{eq:ABC}
\begin{array}{cccc}
  A & \approx & B & C\\
m \times n & & m \times \ell & \ell \times n
\end{array}
\end{equation}
where $\ell \le \min(m,n) = m \wedge n$ is the desired rank. The goal
is to compute $B$ and $C$ such that $A - BC$ is as small as possible.
Typically, one measures the quality of the approximation by taking
either the spectral norm $\|\cdot\|$ (the largest singular value, also
known as the 2 norm) or the Frobenius norm $\|\cdot\|_F$ (the root-sum
of squares of the singular values) of the residual $A-BC$. It is
well-known that the best rank-$\ell$ approximation, measured either in
the spectral or Frobenius norm, is obtained by truncating the singular
value decomposition (SVD), but this can be prohibitively expensive
when dealing with large matrix dimensions.

Recent work \cite{Rokhlin2008} introduced a randomized algorithm for
matrix factorization with lower computational complexity.
\begin{algorithm}
\caption{Randomized algorithm for matrix approximation}
\label{alg}
\begin{algorithmic}[1]
  \Require Input: $m \times n$ matrix $A$ and desired rank $\ell$.
  \Statex Sample an $n \times \ell$ test matrix $G$ with independent
  mean-zero, unit-variance Gaussian entries.   
\Statex Compute $H =  AG$.  
\Statex Construct $Q \in \R^{m \times \ell}$ with columns forming
  an orthonormal basis for the range of $H$.
\Statex \textbf{return} the approximation $B = Q$, $C = Q^*A$.
\end{algorithmic}
\end{algorithm}

The algorithm is simple to understand: $H = AG$ is an approximation of
the range of $A$; we therefore project the columns of $A$ onto this
approximate range by means of the orthogonal projector $Q Q^*$ and
hope that $A \approx BC = QQ^*A$. Of natural interest is the accuracy
of this procedure: how large is the size of the residual? Specifically,
how large is
% \begin{equation}
% \label{eq:resid}
$\|(I-QQ^*)A\|$? 
%\end{equation}
% There are variants of Algorithm \ref{alg}, but the residual error is
% at least \eqref{eq:resid}. 

The subject of the beautiful survey \cite{HMT} as well as
\cite{Rokhlin2008} is to study this problem and provide an analysis of
performance.  Before we state the sharpest results known to date, we
first recall that if $\sigma_1 \geq \sigma_2 \geq \ldots \geq
\sigma_{m \wedge n}$ are the ordered singular values of $A$, then the
best rank-$\ell$ approximation obeys
\begin{equation*}
  \text{min} \{\| A - B \| : \operatorname{rank}(B) \le \ell\} 
= \sigma_{\ell+1}. 
\end{equation*}
It is known that there are choices of $A$ such that $\E \|(I-QQ^*)A\|$
is greater than $\sigma_{\ell+1}$ by an arbitrary multiplicative
factor, see e.g.~\cite{HMT}. For example, setting
\begin{equation*}
A= \begin{bmatrix}t&0\\ 0& 1 \end{bmatrix}
\end{equation*}
and $\ell = 1$, direct computation shows that $\lim_{t\rightarrow
  \infty} \E \|(I-QQ^*)A\| = \infty$. Thus, we write $\ell = k+p$
(where $p>0$) and seek $\underline{b}$ and $\overline{b}$ such that
\begin{equation*}
  \underline{b}(m,n,k,p) \leq \sup_{A \in \R^{m\times n}} \E \| (I - QQ^*)A \|/\sigma_{k+1} \leq \overline{b}(m,n,k,p).
\end{equation*}
% Similarly, we seek to understand the ``typical" behavior of $ \| A_{(k+p)} \|/\sigma_{k+1}$.

We are now ready to state the best results concerning the performance
of Algorithm \ref{alg} we are aware of.
\begin{theorem}[\cite{HMT}]
\label{teo:HMT}
Let $A$ be an $m \times n$ matrix and run Algorithm \ref{alg} with
$\ell= k+p$, then
\begin{equation}
\label{eq:HMT}
\E \|(I - QQ^*)A\|\leq \left[ 1 + \frac{4 \sqrt{k+p}}{p-1}\sqrt{m \wedge n} 
\right] \sigma_{k+1}.
\end{equation}
\end{theorem}
It is a priori unclear whether this upper bound correctly predicts the
expected behavior or not. That is to say, is the dependence upon the
problem parameters in the right-hand side of the right order of
magnitude? Is the bound tight or can it be substantially improved? Are
there lower bounds which would provide ultimate limits of performance?
The aim of this paper is merely to provide some definite answers to
such questions.

% \begin{theorem}[\cite{HMT}]
% \label{teo:HMT_MIXEDNORM}
%  In the setup of Theorem \ref{teo:HMT}, with $\sigma_i$ the $i$th singular value of $A$
% \begin{equation}
% \label{eq:HMT_MIXEDNORM}
% \E \|(I - QQ^*)A\|\leq \left( 1 + \sqrt{\frac{k}{p-1}} \right) \sigma_{k+1} + \frac{e \sqrt{k+p}}{p} \left ( \sum_{j=k+1}^{m \wedge n} \sigma_j^2 \right)^{1/2}.
% \end{equation}
% \end{theorem}

\subsection{Sharp bounds}

It is convenient to write the residual in Algorithm \ref{alg} as
$f(A,G) := (I - QQ^*) A$ as to make the dependence on the random test
matrix $G$ explicit. Our main result states that there is an explicit
random variable whose size completely determines the accuracy of the
algorithm. This statement uses a natural notion of stochastic
ordering; below we write $X \,{\buildrel d \over \ge}\, Y$ if and only
if the random variables $X$ and $Y$ obey $\P(X \ge t) \ge \P(Y\ge t)$
for all $t \in \R$.

\begin{theorem}\label{newtheorem}
  Suppose without loss of generality that $m \ge n$. Then in the setup
  of Theorem \ref{teo:HMT}, for each matrix $A \in \R^{m \times n}$,
\begin{equation*}
\label{eq:upper}
  \|(I - QQ^*)A\|\leqd \sigma_{k+1} \, W,
\end{equation*}
where $W$ is the random variable 
\begin{equation}
\label{eq:W}
W =  \| f(I_{n-k}, X_2) 
\begin{bmatrix} X_1 \Sigma^{-1} & I_{n-k} \end{bmatrix} \|;
\end{equation}
here, $X_1$ and $X_2$ are respectively $(n-k) \times k$ and $(n-k)
\times p$ matrices with i.i.d.~$\mathcal{N}(0,1)$ entries, $\Sigma$ is
a $k \times k$ diagonal matrix with the singular values of a
$(k+p)\times k$ Gaussian matrix with i.i.d.~$\mathcal{N}(0,1)$
entries, and $I_{n-k}$ is the $(n-k)$-dimensional identity
matrix. Furthermore, $X_1$, $X_2$ and $\Sigma$ are all independent
(and independent from $G$). In the other direction, for any $\epsilon
> 0$, there is a matrix $A$ with the property
\begin{equation*}
\label{eq:lower}
 \|(I - QQ^*)A\| \geqd (1-\epsilon) \sigma_{k+1} \, W. 
\end{equation*}
In particular, this gives 
\[
\sup_A \,\, \E \|(I - QQ^*)A\|/\sigma_{k+1} = \E W.
\]
\end{theorem}

The proof of the theorem is in Section
\ref{sec:newbounds_subsection}. To obtain very concrete bounds from
this theorem, one can imagine using Monte Carlo simulations by sampling
from $W$ to estimate the worst error the algorithm
commits. Alternatively, one could derive upper and lower bounds about
$W$ by analytic means. The corollary below is established in Section
\ref{sec:corollaries}.
\begin{corollary}[{\bf Nonasymptotic bounds}]
  \label{cor:newthm} With $W$ as in \eqref{eq:W} (recall
  $n \le m$),
\begin{equation}
\label{eq:newthm}
\sqrt{n-(k+p+2)} \, \E \|\Sigma^{-1}\| \le  \E W \le 1 +\left(\sqrt{n-k}+\sqrt{k}\right) \, \E \|\Sigma^{-1}\|.  
\end{equation}
\end{corollary}

In the regime of interest where $n$ is very large and $k + p \ll n$,
the ratio between the upper and lower bound is practically equal to 1
so our analysis is very tight. Furthermore, Corollary \ref{cor:newthm}
clearly emphasizes why we would want to take $p > 0$. Indeed, when $p
= 0$, $\Sigma$ is square and nearly singular so that both $\E
\|\Sigma^{-1}\|$ and the lower bound become very large. In contrast,
increasing $p$ yields a sharp decrease in $\E \|\Sigma^{-1}\|$ and,
thus, improved performance.

It is further possible to derive explicit bounds by noting (Lemma
\ref{lem:pseudoInverseExpectedValue}) that 
\begin{equation}
\label{eq:Sigma_inv}
\frac{1}{\sqrt{p+1}} \leq \E  \|\Sigma^{-1}\|\leq 
e\frac{\sqrt{k+p}}{p}. 
\end{equation}
Plugging the right inequality into \eqref{eq:newthm} improves upon
\eqref{eq:HMT} from \cite{HMT}.  In the regime where $k + p \ll n$ (we
assume throughout this section that $n \le m$), taking $p = k$, for
instance, yields an upper bound roughly equal to $e \sqrt{2n/k}
\approx 3.84 \sqrt{n/k}$ and a lower bound roughly equal to
$\sqrt{n/k}$, see Figure \ref{fig:empirical}.

\newcommand{\goto}{\rightarrow} When $k$ and $p$ are reasonably large,
it is well-known (see Lemma \ref{lem:almostSureLaws}) that
\[
\sigma_{\text{min}}(\Sigma) \approx \sqrt{k+p} - \sqrt{k}
\]
so that in the regime of interest where $k + p \ll n$, both the lower
and upper bounds in \eqref{eq:newthm} are about equal to
\begin{equation}
  \label{eq:error-proxy}
\frac{\sqrt{n}}{\sqrt{k+p} - \sqrt{k}}. 
\end{equation}
We can formalize this as follows: in the limit of large dimensions
where $n \goto \infty$, $k,p \goto \infty$ with $p/k \goto \rho > 0$
(in such a way that $\limsup\,\, (k + p)/n < 1$), we have almost surely
\begin{equation}
  \label{eq:infinite-upper}
  \limsup\,\,  \frac{W}{\overline{b}(n,k,p)} \le 1,  \qquad \overline{b}(n,k,p) = \frac{\sqrt{n-k} + \sqrt{k}}{\sqrt{k + p} - \sqrt{k}}. 
\end{equation}
Conversely, it holds almost surely that 
\begin{equation}
  \label{eq:infinite-lower}
  \liminf \,\,  \frac{W}{\underline{b}(n,k,p)} \ge 1,  \qquad \underline{b}(n,k,p) = \frac{\sqrt{n - k - p}}{\sqrt{k + p} - \sqrt{k}}. 
\end{equation}
A short justification of this limit behavior may also be found in
Section \ref{sec:corollaries}.

% We can additionally answer that the bounds mixed-norm bounds provided in \cite{HMT} are tight up to a constant multiple.
% \begin{corollary}[\bf{Mixed norm lower-bound}]
% \label{cor:OUR_MIXEDNORM}
%  In the setup of Theorem \ref{teo:HMT}, with $\sigma_i$ the $i$th singular value of $A$, there exist choices of $A$ such that,
% \begin{equation}
% \label{eq:OUR_MIXEDNORM}
% \E \|(I - QQ^*)A\|\geq \sqrt{\frac{2}{\pi}} \sqrt{\frac{n-(k+p)}{ p(n-k)}}  \left ( \sum_{j=k+1}^{m \wedge n} \sigma_j^2 \right)^{1/2}.
% \end{equation}
% \end{corollary}

% This result is close to matching the upper bound from Theorem \ref{teo:HMT_MIXEDNORM}.
% \rw{The real weakness is that $\frac{1}{\sqrt{p}}$ and $\frac{1}{\sqrt{k+p} - \sqrt{k}}$ might get far apart, especially if one considers $p$ fixed, $k \rightarrow \infty$}.

\subsection{Innovations}
\label{sec:innovation}

Whereas the analysis in \cite{HMT} uses sophisticated concepts and
tools from matrix analysis and from perturbation analysis, our method
is different and only uses elementary ideas (for instance, it should
be understandable by an undergraduate student with no special
training). In a nutshell, the authors in \cite{HMT} control the error
of Algorithm \ref{alg} by establishing an upper bound about
$\|f(A,X)\|$ holding for all matrices $X$ (the bound depends on
$X$). From this, they deduce bounds about $\|f(A,G)\|$ in expectation
and in probability by integrating with respect to $G$. A limitation of
this approach is that it does not provide any estimate of how close
the upper bound is to being tight.

In contrast, we perform a sequence of reductions, which ultimately
identifies the worst-case input matrix. The crux of this reduction is
a monotonicity property, which roughly says that if the spectrum of a
matrix $A$ is larger than that of another matrix $B$, then the
singular values of the residual $f(A,G)$ are stochastically greater
than those of $f(B,G)$, see Lemma
\ref{lem:monotonicity_forward_statement} in Section
\ref{sec:monotonicity_forward_statement} for details. Hence, applying
the algorithm to $A$ results in a larger error than when the algorithm
is applied to $B$. In turn, this monotonicity property allows us to
write the worst-case residual in a very concrete form. With this
representation, we can recover the deterministic bound from \cite{HMT}
and immediately see the extent to which it is sub-optimal. Most
importantly, our analysis admits matching lower and upper bounds as
discussed earlier.

Our analysis of Algorithm \ref{alg}, presented in Section
\ref{sec:newbounds_subsection}, shows that the approximation error is
heavily affected by the spectrum of the matrix $A$ past its first
$k+1$ singular values.\footnote{To accommodate this, previous works
  also provide bounds in terms of the singular values of $A$ past
  $\sigma_{k+1}$.}  In fact, suppose $m \ge n$ and let $D_{n-k}$ be
the diagonal matrix of dimension $n-k$ equal to
$\operatorname{diag}(\sigma_{k+1}, \sigma_{k+2}, \ldots,
\sigma_n)$. Then our method show that the worst case error for
matrices with this tail spectrum is equal to the random variable
\[
W(D_{n-k}) =  \| f(D_{n-k}, X_2) 
\begin{bmatrix} X_1 \Sigma^{-1} & I_{n-k} \end{bmatrix} \|. 
\]
In turn, a very short argument gives the expected upper bound below: %similar
% to Theorem 10.6 in \cite{HMT}.
\begin{theorem}
 \label{cor:OUR_MIXEDNORM} 
 Take the setup of Theorem \ref{teo:HMT} and let $\sigma_i$ be the
 $i$th singular value of $A$. Then
\begin{equation}
\label{eq:mixed}
\E \|(I - QQ^*)A\| \le \Bigl(1 + \sqrt{\frac{k}{p-1}}\Bigr) \sigma_{k+1} + \E \|\Sigma^{-1}\|\, \sqrt{\sum_{i > k} \sigma_i^2}. 
\end{equation}
Substituting $\E \|\Sigma^{-1}\|$ with the upper bound in
\eqref{eq:Sigma_inv} recovers Theorem 10.6 from \cite{HMT}.
\end{theorem}

This bound is tight in the sense that setting $\sigma_{k+1} =
\sigma_{k+2} = \ldots = \sigma_n = 1$ essentially yields the upper
bound from Corollary \ref{cor:newthm}, which as we have seen, cannot
be improved. % \ejc{Strangely enough, it is a tiny bit better.}

\subsection{Experimental results}
\label{numerics}

To examine the tightness of our analysis of performance, we apply
Algorithm \ref{alg} to the `worst-case' input matrix and compute the
spectral norm of the residual, performing such computations for fixed
values of $m$, $n$, $k$ and $p$.  We wish to compare the sampled
errors with our deterministic upper and lower bounds, as well as with
the previous upper bound from Theorem \ref{teo:HMT} and our error
proxy \eqref{eq:error-proxy}.  Because of Lemma \ref{MANINV}, the
worst-case behavior of the algorithm does not depend on $m$ and $n$
separately but only on $\min(m,n)$. Hence, we set $m=n$ in this
section.

% On the $x$-axis, we vary $m$ from $10^4$ to $10^5$ and hold $k$ and
% $p$ to be $1\%$ of $m$ in Figure \ref{kpfixedratio}. Each red dot
% represents the $2$ norm error of one run of Algorithm \ref{alg} on the
% ``worst-case" matrix. The lines are bounds on $2$-norm for the
% selected value of $m$, $k$ and $p$.  The purple line plots the
% previous upper bound from Theorem \ref{teo:HMT}. The green line is the
% upper bound and the blue is the lower bound from our new nonasymptotic
% result Corollary \ref{cor:newthm}.  The yellow line is our suggested
% ``rule of thumb" from Equation \ref{eq:error-proxy}.

\begin{figure}[h!]
        \centering
        \begin{subfigure}[b]{0.70\textwidth}
                \centering
 \caption{$k = p = 0.01 \, n$}
                \includegraphics[width=\textwidth]{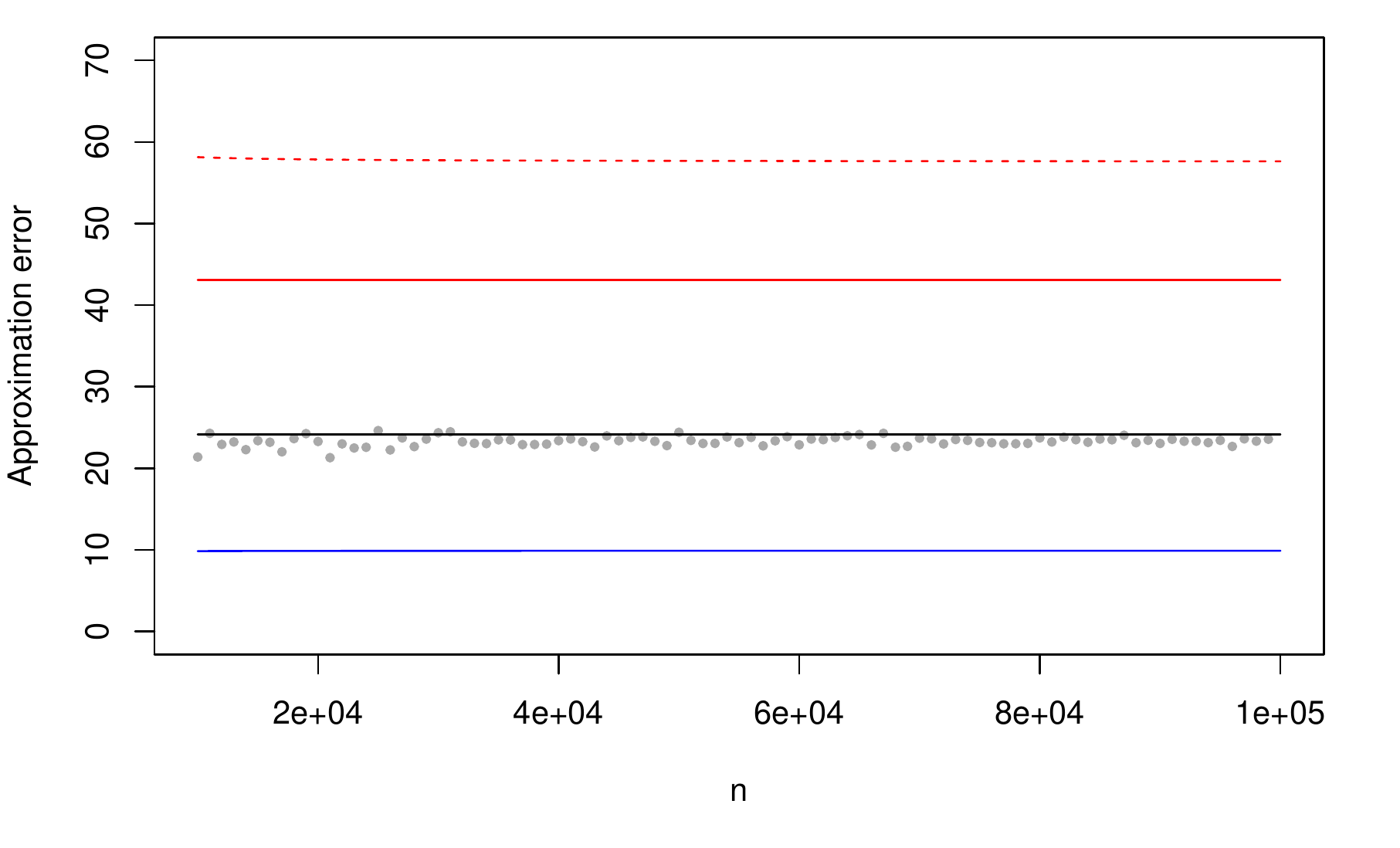}
%{./images/fixedRatio.pdf}
                               \label{fig:kpfixedratio}
        \end{subfigure}
                \begin{subfigure}[b]{0.70\textwidth}
                \centering
 \caption{$k = p = 100$}
                \includegraphics[width=\textwidth]{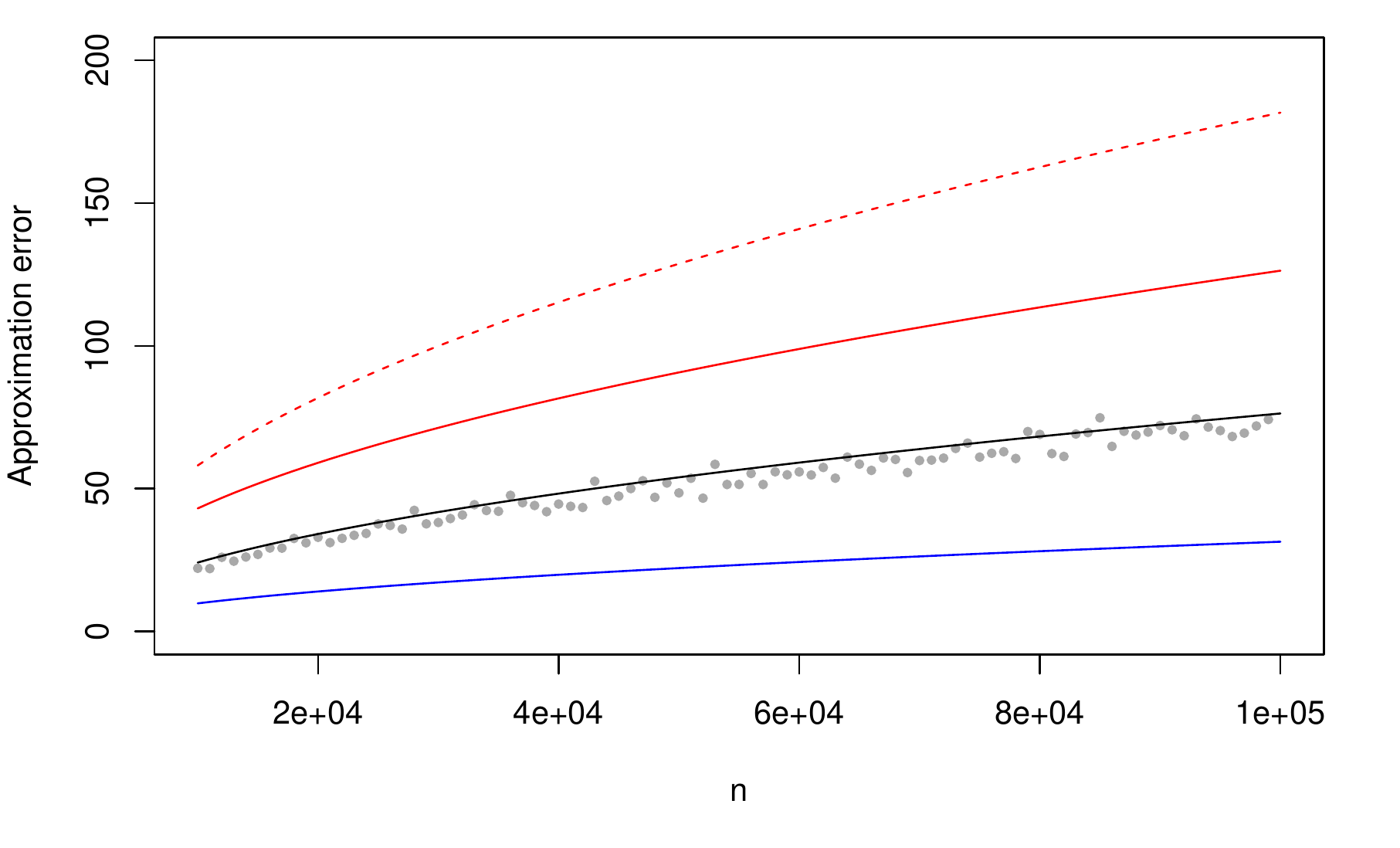}
% {./images/variedRatio.pdf}
                               \label{fig:kpvariedratio}
        \end{subfigure}
        \caption{Spectral norm $\|(I-QQ^*)A\|$ of the residual with
          worst-case matrix of dimension $n \times n$ as input with
          $n$ varying between $10^4$ and $10^5$. Each grey dot
          represents the error of one run of Algorithm \ref{alg}. The
          lines are bounds on the spectral norm: the red dashed line
          plots the previous upper bound \eqref{eq:HMT}. The red
          (resp.~blue) solid line is the upper (resp.~lower) bound
          combining Corollary \ref{cor:newthm} and
          \eqref{eq:Sigma_inv}.  The black line is the error proxy
          \eqref{eq:error-proxy}. In the top plot, $k = p = 0.01 \,
          n$.  Keeping fixed ratios results in constant error. Holding
          $k$ and $p$ fixed in the bottom plot while increasing $n$
          results in approximations of increasing error.}
          \label{fig:empirical}
\end{figure}

Figure \ref{fig:empirical} reveals that the new upper and lower bounds
are tight up to a small multiplicative factor, and that the previous
upper bound is also fairly tight. Further, the plots also demonstrate
the effect of concentration in measure---the outcomes of different
samples each lie remarkably close to the yellow rule of thumb,
especially for larger $n$, suggesting that for practical
purposes the algorithm is deterministic.  Hence, these experimental
results reinforce the practical accuracy of the error proxy
\eqref{eq:error-proxy} in the regime $k+p \ll n$ since we can see that
the worst-case error is just about \eqref{eq:error-proxy}.

\begin{figure}[h!]
        \centering
        \begin{subfigure}[b]{0.49\textwidth}
                \centering
                \caption{$m=n=10^5, k = p = 10^2$}
\vspace*{-1cm}
                \includegraphics[width=\textwidth]{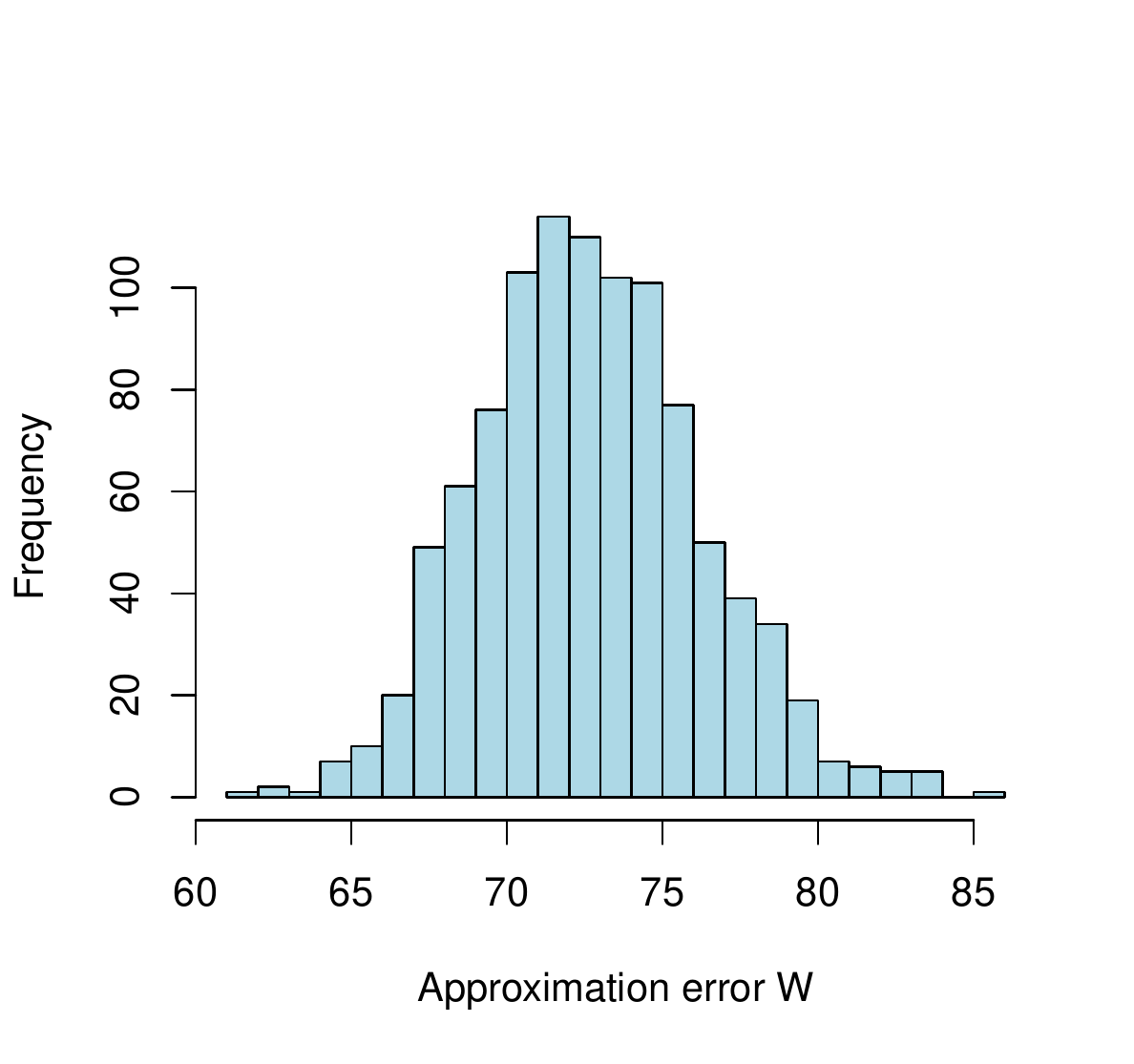}
                %\caption{$m=10^5, k = p = 10^2$}
                \label{fig:samples,n=1e5,k=1e2,p=1e2}
        \end{subfigure}
        \begin{subfigure}[b]{0.49\textwidth}
                \centering
                \caption{$m=n=10^5, k = p = 10^3$}
\vspace*{-1cm}
                \includegraphics[width=\textwidth]{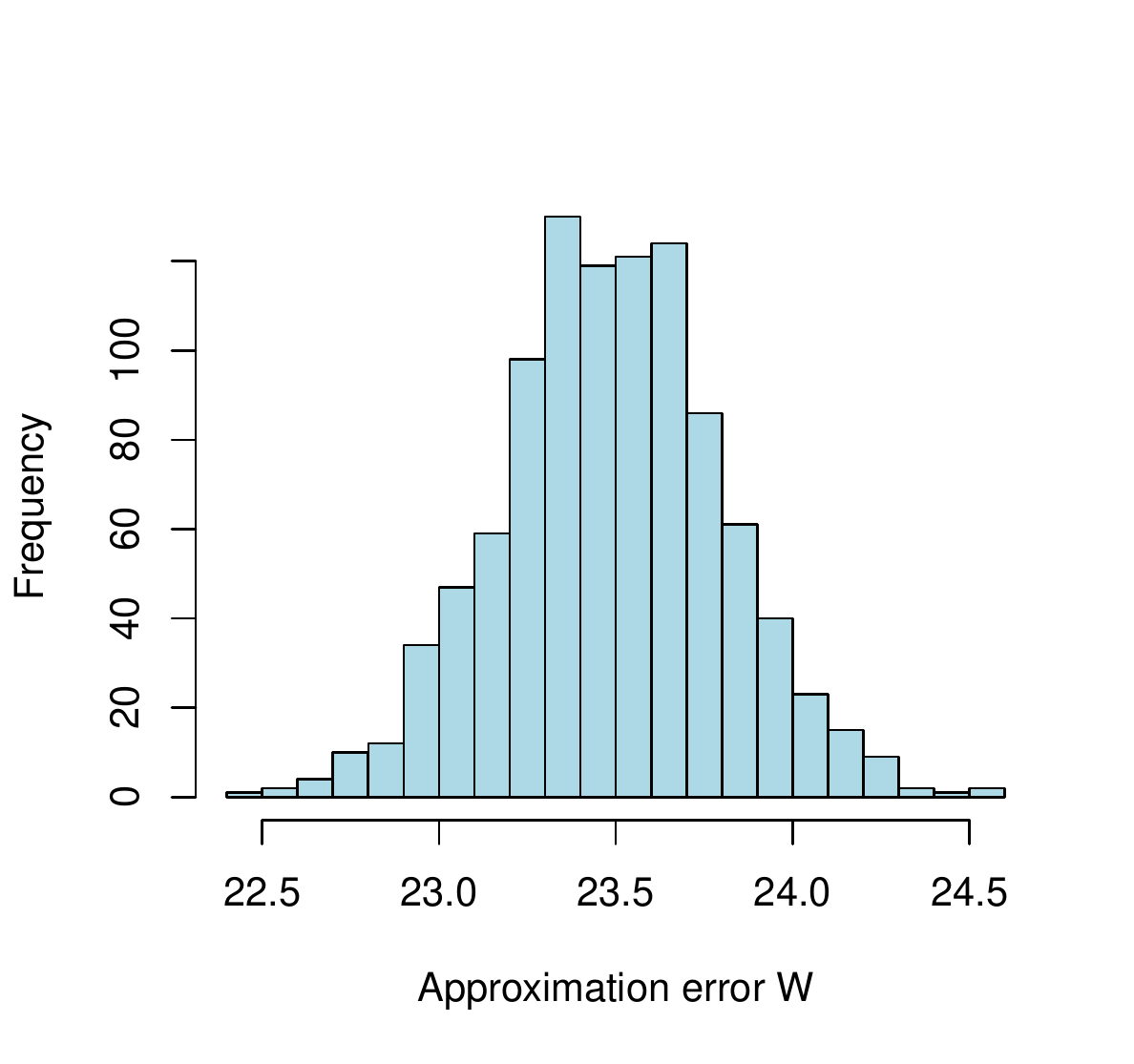}
                % \caption{$m=10^5, k = p = 10^3$}
                \label{fig:samples,n=1e5,k=1e3,p=1e3}
        \end{subfigure}
%\vspace*{-0.5cm}
        \caption{We fix $n$ (recall that $m = n$), $k$ and $p$ and
          plot the approximation error $W$ across $1000$ independent
          runs of the algorithm.  For the larger value of $k$ and $p$,
          we see reduced variability---the approximation errors range
          from about $61$ to $85$ for $k=p=10^2$ and $22.5$ to $24.5$
          for $k=p=10^3$, much less variability in both absolute and
          percentage terms. Similarly, the empirical standard deviation
          with $k=p=10^2$ is approximately $3.6$ while with $k=p=10^3$
          it falls to $.32$. }
          \label{fig:empirical_samples}
\end{figure}

Figure \ref{fig:empirical_samples} gives us a sense of the variability
of the algorithm for two fixed values of the triple $(n, k, p)$.  As
expected, as $k$ and $p$ grow the variability of the algorithm
decreases, demonstrating the effect of concentration of measure.

\subsection{Complements}
\label{sec:complements}

In order to make the paper a little more self-contained, we briefly
review some of the literature as to explain the centrality of the
range finding problem (Algorithm \ref{alg}). For instance, suppose we
wish to construct an approximate SVD of a very large matrix $m \times
n$ matrix $A$. Then this can be achieved by running Algorithm
\ref{alg} and then computing the SVD of the `small' matrix $C = Q^*A$,
check Algorithm \ref{alg-svd} below, which returns an approximate SVD
$A \approx U \Sigma V^*$. Assuming the extra steps in Algorithm
\ref{alg-svd} are exact, we have
\begin{equation}
  \label{eq:svd}
  \|A - U \Sigma V^*\| = \|A - Q\hat U \Sigma V^*\| = \|A - QC\| = \|A - QQ^*A\|
\end{equation}
so that the approximation error is that of Algorithm \ref{alg} whose
study is the subject of this paper.

\begin{algorithm}
\caption{Randomized algorithm for approximate SVD computation}
\label{alg-svd}
\begin{algorithmic}[1]
  \Require Input: $m \times n$ matrix $A$ and desired rank $\ell$.
  \Statex Run Algorithm \ref{alg}.
\Statex Compute $C =  Q^*A$.
\Statex Compute the SVD of $C = \hat U \Sigma V^*$.
\Statex \textbf{return} the approximation $U = Q\hat U$, $\Sigma$, $V$.
\end{algorithmic}
\end{algorithm}

The point here of course is that since the matrix $C = Q^*A$ is $\ell
\times n$---we typically have $\ell \ll n$---the computational cost of
forming its SVD is on the order of $O(\ell^2 n)$ flops and fairly
minimal. (For reference, we note that there is an even more effective
single-pass variant of Algorithm \ref{alg-svd} in which we do not
need to re-access the input matrix $A$ once $Q$ is available, please
see \cite{HMT} and references therein for details.)

Naturally, we may wish to compute other types of approximate matrix
factorizations of $A$ such as eigenvalue decompositions, QR
factorizations, interpolative decompositions (where one searches for
an approximation $A \approx BC$ in which $B$ is a subset of the
columns of $A$), and so on. All such computations would follow the
same pattern: (1) apply Algorithm \ref{alg} to find an approximate
range, and (2) perform classical matrix factorizations on a matrix of
reduced size. This general strategy, namely, approximation followed by
standard matrix computations, is of course hardly anything new. To be
sure, the classical Businger-Golub algorithms for computing partial QR
decompositions follows this pattern. Again, we refer the interested
reader to the survey \cite{HMT}.

We have seen that when the input matrix $A$ does not have a rapidly
decaying spectrum, as this may be the case in a number of data analysis
applications, the error of approximation Algorithm \ref{alg}
commits---the random variable $W$ in Theorem \ref{newtheorem}---may be
quite large. In fact, when the singular values hardly decay at all, it
typically is on the order of the error proxy
\eqref{eq:error-proxy}. This results in poor performance. On the other
hand, when the singular values decay rapidly, we have seen that the
algorithm is provably accurate, compare Theorem
\ref{cor:OUR_MIXEDNORM}.  This suggests using a power iteration,
similar to the block power method, or the subspace iteration in
numerical linear algebra: Algorithm \ref{alg-power} was proposed in
\cite{RokhlinTygert09}.

\begin{algorithm}
  \caption{Randomized algorithm with power trick for matrix
    approximation}
\label{alg-power}
\begin{algorithmic}[1]
  \Require Input: $m \times n$ matrix $A$ and desired rank $\ell$.
  \Statex Sample an $n \times \ell$ test matrix $G$ with independent
  mean-zero, unit-variance Gaussian entries.   
\Statex Compute $H =  (AA^*)^q AG$.  
\Statex Construct $Q \in \R^{m \times \ell}$ with columns forming
  an orthonormal basis for the range of $H$.
\Statex \textbf{return} the approximation $B = Q$, $C = Q^*A$.
\end{algorithmic}
\end{algorithm}

The idea in Algorithm \ref{alg-power} is of course to turn a slowly
decaying spectrum into a rapidly decaying one at the cost of more
computations: we need $2q+1$ matrix-matrix multiplies instead of just
one. The benefit is improved accuracy. Letting $P$ be any orthogonal
projector, then a sort of Jensen inequality states that for any matrix
$A$, 
\[
\|PA\| \le \|P(AA^*)^{q}A\|^{1/(2q+1)}, 
\] 
see \cite{HMT} for a proof.  Therefore, if $Q$ is computed via the
power trick (Algorithm \ref{alg-power}), then 
\[
\|(I-QQ^*)A\| \le \|(I-QQ^*)(AA^*)^{q}A\|^{1/(2q+1)}.
\]
This immediately gives a corollary to Theorem \ref{newtheorem}. 
\begin{corollary}
\label{cor:power}
Let $A$ and $W$ be as in Theorem \ref{newtheorem}. Applying Algorithm
\ref{alg-power} yields
\begin{equation}
\label{eq:power}
  \|(I - QQ^*)A\|\leqd \sigma_{k+1} \, W^{1/(2q+1)}. 
\end{equation}
\end{corollary}

It is amusing to note that with, say, $n = 10^9$, $k = p = 200$, and
the error proxy \eqref{eq:error-proxy} for $W$, the size of the error
factor $W^{1/(2q+1)}$ in \eqref{eq:power} is about 3.41 when $q = 3$.

Further, we would like to note that our analysis exhibits a sequence
of matrices \eqref{eq:key} that have approximation errors that limit
to the worst case approximation error when $q=0$.  However, when
$q=1$, this same sequence of matrices limits to having an
approximation error exactly equal to $1$, which is the best possible
since this is the error achieved by truncating the SVD. It would be
interesting to study the tightness of the upper bound \eqref{eq:power}
and we leave this to future research.

% \ejcrw{Now you write that this suggests that our bound is not tight. I
%   don't agree with this. Numerical experiments do not show an error of
%   1 so what this really means is that M(t) is no longer the worst-case
%   input. An important question remains though. We have shown that when
%   $q = 0$, we cannot do better than $(n/k)^{1/2}$. When $q = 1$, we do
%   at least $(n/k)^{1/6}$. The question is whether this is the limit or
%   whether this is very pessimistic---as you write---and the error
%   should really be constant. What do you think?} \rw{I agree with your
%   summary of the issue and do not have a strong suspicion either way
%   about if we're tight in the $q>0$ case.  All I meant is that this
%   finding raises the possibility that we are not tight in the $q>0$
%   case.}

\subsection{Notation}

In Section \ref{sec:discussion}, we shall see that among all possible
test matrices, Gaussian inputs are in some sense optimal. Otherwise,
the rest of the paper is mainly devoted to proving the novel results
we have just presented.  Before we do this, however, we pause to
introduce some notation that shall be used throughout. We reserve
$I_n$ to denote the $n \times n$ identity matrix. When $G$ is an $n
\times \ell$ random Gaussian matrix, we write $A_{(\ell)} = f(A,G)$ to
save space. Hence, $A_{(\ell)}$ is a random variable.

We use the partial ordering of $n$-dimensional vectors and write $x
\ge y$ if $x - y$ has nonnegative entries.  Similarly, we use the
semidefinite (partial) ordering and write $A \succeq B$ if $A-B$ is
positive semidefinite. We also introduce a notion of stochastic
ordering in $\R^n$: given two $n$-dimensional random vectors $z_1$ and
$z_2$, we say that $z_1 \geqd z_2$ if $\P(z_1 \geq x) \geq \P(z_2 \geq
x)$ for all $x \in \R^n$. If instead $\P(z_1 \geq x) = \P(z_2 \geq x)$
then we say that $z_1$ and $z_2$ are equal in distribution and write
$z_1 \eqd z_2$. A function $g: \R^n \rightarrow \R$ is monotone
non-decreasing if $x \geq y$ implies $g(x) \geq g(y)$. Note that if
$g$ is monotone non-decreasing and $z_1 \geqd z_2$, then $\E g(z_1)
\geq \E g(z_2)$.

\section{Proofs}
\label{sec:newbounds}

\subsection{Monotonicity}
\label{sec:monotonicity_forward_statement}

The key insight underlying our analysis is this:
\begin{lemma}[{\bf Monotonicity}] 
  \label{lem:monotonicity_forward_statement}
  If $A$ and $B$ are matrices of the same dimensions obeying
  $\sigma(B) \leq \sigma(A)$, then 
\[
\sigma\bigl(B_{(\ell)} \bigr) \leqd
  \sigma \bigl(A_{(\ell)} \bigr). 
\]
\end{lemma} 
The implications of this lemma are two-fold. First, to derive error
estimates, we can work with diagonal matrices without any loss of
generality. Second and more importantly, it follows from the
monotonicity property that
\begin{equation}
\label{eq:key}
\sup_{A \in \R^{m\times n}} \E \| f(A,G) \| / \sigma_{k+1} = \lim_{t \rightarrow \infty} \E \| f(M(t), G)\|,  \quad M(t) = \begin{bmatrix}tI_k&0\\ 0&  
I_{n-k}\end{bmatrix}. 
\end{equation}
The rest of this section is thus organized as follows:
\begin{itemize}
\item The proof of the monotonicity lemma is in Section
  \ref{sec:monotonicity_proof_section}.

\item To prove Theorem \ref{newtheorem}, it suffices to consider a
  matrix input as $M(t)$ in \eqref{eq:key} with $t \rightarrow
  \infty$. This is the object of Section
  \ref{sec:newbounds_subsection}.

\item Bounds on the worst case error (the proof of Corollary
  \ref{cor:newthm}) are given in Section \ref{sec:corollaries}.

\item The mixed norm Theorem \ref{cor:OUR_MIXEDNORM} is proved in
  Section \ref{sec:mixed}.
\end{itemize}
Additional supporting materials are found in the Appendix.

\subsection{Proof of the monotonicity property}
\label{sec:monotonicity_proof_section}

We begin with an intuitive lemma. 
\begin{lemma}[{\bf Rotational invariance}] 
  \label{MANINV}
  Let $A \in \R^{m \times n}$, $U \in \R^{m \times m}$, $V \in \R^{n
    \times n}$ be arbitrary matrices with $U$ and $V$ orthogonal. Then
  $\sigma(A_{(\ell)}) \eqd \sigma((UAV)_{(\ell)})$. In particular, if
  $A = U \Sigma V^*$ is a singular value decomposition of $A$, then
  $\|A_{(\ell)} \| \eqd \|\Sigma_{(\ell)} \|$.
\end{lemma}
\begin{proof}
  Observe that by spherical symmetry of the Gaussian distribution, a
  test vector $g$ (a column of $G$) has the same distribution as $r =
  Vg$. We have
\[
(UAV)_{(1)}=UAV - \frac{UAVg (UAVg)^* UAV}{\|UAVg\|_2^2}
\]
and thus,
\begin{align*}
  (UAV)_{(1)}  &= UAV - \frac{UA r r^* A^* A V }{\|Ar\|_2^2} \\
  &= U \left( A - \frac{A r r^* A^* A}{\|Ar\|_2^2} \right) V\\
  &\eqd U A_{(1)} V.
\end{align*}
Hence, $\sigma((UAV)_{(1)}) \eqd \sigma(U A_{(1)} V) = \sigma(
A_{(1)})$, which means that the distribution of $\sigma(A_{(1)})$
depends only upon $\sigma(A)$. By induction, one establishes
$\sigma((UAV)_{(\ell)}) \eqd \sigma( U A_{(\ell)} V) =
\sigma(A_{(\ell)})$ in the same fashion. The induction step uses Lemma
\ref{lem:chaining}, which states that $f(A, G) = f(f(A, G_1), G_2)$,
where $G_1$ and $G_2$ are a partition of the columns of $G$.
\end{proof}

A consequence of this lemma is that we only need to show the
monotonicity property for pairs of diagonal matrices obeying $\Sigma_1
\succeq \Sigma_2$. The lemma below proves the monotonicity property in
the special case where $\ell = 1$.
\begin{lemma} \label{deterministicinequality} Let $\Sigma_1$ and
  $\Sigma_2$ be $n \times n$ diagonal and positive semidefinite
  matrices obeying $\Sigma_1 \succeq \Sigma_2$. Let $g$ be an
  arbitrary vector in $\R^n$. Then for all $x$, 
\[
\|f(\Sigma_1,g) x\|_2 \geq \|f(\Sigma_2, g) x\|_2.
\] 
This implies that $\sigma \left(f(\Sigma_1, g) \right) \geq \sigma
\left( f(\Sigma_2, g) \right)$.
\end{lemma}
\begin{proof}
  Introduce
\begin{equation*}
\begin{array}{llcl}
h: & \R_+^n & \rightarrow & \R\\
& \sigma  & \mapsto & \| f(\operatorname{diag}(\sigma), g) x \|_2^2
\end{array}
\end{equation*}
in which $\operatorname{diag}(\sigma)$ is the diagonal matrix with
$\sigma_i$ on the diagonal. Clearly, it suffices to show that
$\partial_i h(\sigma) \ge 0$ to prove the first claim. Putting $\Sigma =
\operatorname{diag}(\sigma)$, $h(\sigma)$ is given by 
\begin{equation*}
  h(\sigma) =  x^* \left( \Sigma \left ( I - \frac{(\Sigma g) (\Sigma g)^*}{ \| \Sigma g \|_2^2 } \right ) \Sigma \right)x = \sum_k \sigma_k^2 x_k^2 - \frac{\left( \sum_{k=1}^n \sigma_k^2 g_k x_k \right)^2}{\sum_{k=1}^n \sigma_k^2 g_k^2 }. \label{residual}
\end{equation*}
Taking derivatives gives 
\begin{equation*}
  \frac{\partial}{\partial \sigma_i} h(\sigma) = 2\sigma_i x_i^2 
  - 4 \sigma_i x_i t_i  + 2\sigma_i t_i^2, \qquad t_i = g_i \, \frac{ \sum_{k=1}^n \sigma_k^2 g_k x_k}{ \sum_{k=1}^n \sigma_k^2 g_k^2}.
\end{equation*}
Hence, $\partial_i h(\sigma) = 2 \sigma_i(x_i-t_i)^2\geq 0$.

The second part of the lemma follows from Lemma
\ref{hornandjohnsoninequality} below, whose result is a consequence of
Corollary 4.3.3 in \cite{HornAndJohnson} and the following fact:
$\|Bx\|_2^2 \leq \|Ax\|_2^2$ for all $x$ if and only if $B^* B \preceq
A^* A$.
\end{proof}

\begin{lemma} \label{hornandjohnsoninequality}
If $\|Bx\|_2^2 \leq \|Ax\|_2^2$ for all $x$, then $\sigma(B) \leq \sigma(A)$.
\end{lemma}

We now extend the proof to arbitrary $\ell$, which finishes the proof
of the monotonicity property.
\begin{lemma} \label{lem:strongestmonotonicity} Take $\Sigma_1$ and
  $\Sigma_2$ as in Lemma \ref{deterministicinequality} and let $G\in
  \R^{n \times \ell}$ be a test matrix. Then for all $x \in \R^n$,
\[
  \|f(\Sigma_1,G)x\|_2 \geq \| f(\Sigma_2, G) x\|_2.
\]
Again, this implies that $\sigma \left(f(\Sigma_1, G) \right) \geq \sigma
\left( f(\Sigma_2, G) \right)$.
\end{lemma}
\begin{proof}
  Fix $x \in \R^n$ and put $z = \Sigma_1 x$.  The vector $z$ uniquely
  decomposes as the sum of $z^\perp$ and $z^\|$, where $z^\|$ belongs
  to the range of $\Sigma_1 G$ and $z^\perp$ is the orthogonal
  component.  Now, since $z^\|$ is in the range of $\Sigma_1 G$, there
  exists $g$ such that $\Sigma_1 g = z^\|$. We have 
\[
f(\Sigma_1, G) x = z^\perp = f(\Sigma_1, g) x. 
\]
Therefore, Lemma \ref{deterministicinequality} gives
\begin{equation*}
  \|f(\Sigma_1, G) x \|_2=\| f(\Sigma_1, g)x \|_2 \geq \| f(\Sigma_2, g)x \|_2  \geq \| f(\Sigma_2, G)x \|_2. 
\end{equation*}
The last inequality follows from $g$ being in the range of $G$. In
details, let $P_G$ (resp.~$P_g$) be the orthogonal projector onto the
range of $\Sigma_2 G$ (resp.~$\Sigma_2 g$). Then Pythagoras' theorem gives 
\begin{align*}
  \| f(\Sigma_2, g)x \|^2_2 = \|(I - P_g)\Sigma_2 x\|_2^2 & = \|(I -
  P_G)\Sigma_2 x\|_2^2 + \|(P_G - P_g)\Sigma_2 x\|_2^2 \\ & = \|
  f(\Sigma_1, G)x \|^2_2 + \|(P_G - P_g)\Sigma_2 x\|_2^2.
\end{align*}
The second part of the lemma is once more a consequence of Lemma
\ref{hornandjohnsoninequality}.
\end{proof}

\subsection{Proof of Theorem \ref{newtheorem}}
\label{sec:newbounds_subsection}

% \begin{lemma}\label{lem:testINV}
%   Suppose that $G = U \tilde G V^*$, where $U$ and $V$ are orthonormal
%   matrices such that $U$ commutes with $A$, i.e.~$AU=UA$. Then $\|
%   f(A,G)\|=\|f(A,\tilde G)\|$ in both the spectral and Frobenius
%   norms.
% \end{lemma}
% \begin{proof}
%   Let $\qr(A)$ be a function returning an orthobasis for the range of
%   $A$ (for instance obtained via Gram-Schmidt
%   orthogonalization). Setting $Q = \qr(A U \tilde G V^*)$, we can
%   arrange things such that $Q = U\qr( A \tilde G)$. Letting $\| \cdot \|$
%   be either the 2-norm or the Frobeninus norm, we have 
% \begin{align*}
% \| f(A,G) \| &= \| A -  Q Q^* A \| \\
% &= \|A - U \qr(A \tilde G) \qr(A \tilde G)^* U^* A\| \\
% &= \|AU - U \qr(A \tilde G) \qr(A \tilde G)^* U^* A U\| \\
% &= \|UA  - U\qr(A \tilde G) \qr(A \tilde G)^* U^* U A \| \\
% &= \| A - \qr(A\tilde G) \qr(A \tilde G)^* A \| \\
% &= \| f(A, \tilde G) \|.
% \end{align*}
% \end{proof}

We let $D_{n-k}$ be an $(n-k)$-dimensional diagonal matrix and work
with
\begin{equation}
\label{eq:Sigmat}
  M(t) =  \begin{bmatrix}tI_k&0 \\ 0& D_{n-k}\end{bmatrix}. 
\end{equation}
Set $G \in \R^{n \times (k+p)}$ and partition the rows as 
\[
G=\begin{bmatrix}G_1\\
  G_2\end{bmatrix}, \qquad \begin{array}{l}G_1 \in \R^{k\times(k+p)}\\ 
G_2 \in \R^{
  (n-k)\times (k+p)}
\end{array}.
\]
Next, introduce an SVD for $G_1$ 
\[
G_1 = U \begin{bmatrix} \Sigma & 0 \end{bmatrix} V^*, \quad 
U \in \R^{k\times k}, \quad V \in \R^{(k+p) \times (k+p)}, \quad
\Sigma \in \R^{k \times k}, 
\]
and partition $G_2 V$ as
\[
G_2 V = \begin{bmatrix} X_1 & X_2 \end{bmatrix}, \quad X_1 \in
\R^{(n-k) \times k}, \quad X_2 \in \R^{(n-k) \times p}.
\] 
A simple calculation shows that 
\[
H = M(t) G = \begin{bmatrix} tG_1 \\
  D_{n-k} G_2\end{bmatrix} = \begin{bmatrix} U & 0 \\ D_{n-k} X_1 \Sigma^{-1}/t &
  D_{n-k} X_2 \end{bmatrix} \, \begin{bmatrix} t\Sigma & 0 \\ 0 &
  I_p \end{bmatrix} V^*. 
\]
Hence, $H$ and $\begin{bmatrix} U & 0 \\ D_{n-k}X_1 \Sigma^{-1}/t &
  D_{n-k}X_2 \end{bmatrix}$ have the same column space. If $Q_2$ is an
orthonormal basis for the range of $D_{n-k}X_2$, we conclude that
\[
\begin{bmatrix} U & 0 \\ D_{n-k}X_1 \Sigma^{-1}/t & Q_2 \end{bmatrix}
\text{ and, therefore, } \tilde H = \begin{bmatrix} U & 0 \\ (I -
  Q_2 Q_2^*) D_{n-k}X_1 \Sigma^{-1}/t & Q_2 \end{bmatrix}
\]
have the same column space as $H$. Note that the last $p$ columns of
$\tilde H$ are orthonormal.

Continuing, we let $B$ be the first $k$ columns of $\tilde H$. Then
Lemma \ref{appendix:almostorthogonal} below allows us to decompose $B$
as
\begin{equation*}
B = Q + E(t),
\end{equation*}
where $Q$ is orthogonal with the same range space as $B$ and $E(t)$
has a spectral norm at most $O(1/t^2)$. Further since the first $k$
columns of $\tilde H$ are orthogonal to the last $p$ columns, we have
\[
\tilde H = \tilde Q + \tilde E(t), 
\]
where $\tilde Q$ is orthogonal with the same range space as $\tilde H$
and $\tilde E(t)$ has a spectral norm also at most $O(1/t^2)$. This
gives
\begin{align*}
  \lim_{t \goto \infty}  \, M(t) - QQ^* M(t) & =  \lim_{t \goto \infty} \, M(t) - (\tilde H - \tilde E(t))  (\tilde H - \tilde E(t))^* M(t) \\
  & =    \lim_{t \goto \infty} \, M(t) - \tilde H \tilde H^* M(t) \\
  & = \begin{bmatrix} 0 & 0 \\
    -(I - Q_2 Q_2^*) D_{n-k} X_1 \Sigma^{-1} U^* & (I - Q_2 Q_2^*) D_{n-k}  \end{bmatrix}.
\end{align*}
We have reached the conclusion
\begin{equation}
  \label{eq:keytoall}
  \lim_{t \rightarrow \infty} \|f(M(t), G)\| = \|f(D_{n-k}, X_2) \begin{bmatrix}  X_1 \Sigma^{-1} & I_{n-k} \end{bmatrix} \|.
\end{equation}
When $D_{n-k} = I_{n-k}$, this gives our theorem.

\begin{lemma}\label{appendix:almostorthogonal}
  Let $A=\begin{bmatrix} I_k \\
    t^{-1}B \end{bmatrix} \in \R^{n \times k}$. Then $A$ is $O(t^{-2}
  \|B^*B\|_F)$ in Frobenius norm away from a matrix with the same
  range and orthonormal columns.
\end{lemma}
\begin{proof}
  We need to construct a matrix $E \in \R^{(n+k) \times k}$ obeying
  $\|E\|_F = O(\|B^* B\|/t^2)$ and such that (1) $Q = A + E$ is
  orthogonal and (2) $Q$ and $A$ have the same range. Let $U \Sigma
  {V}^*$ be a reduced SVD decomposition for $A$,
\begin{equation*}
  A = U \Sigma  V^* = U {V}^* +  U(\Sigma - I) {V}^* := Q + E. 
\end{equation*}
Clearly, $Q$ is an orthonormal matrix with the same range as
$A$. Next, the $i$th singular value of $A$ obeys
\begin{equation*}
  \sigma_i(A) = 
  \sqrt{\lambda_i(A^* A)} = \sqrt{\lambda_i(I_k + t^{-2} B^* B)} 
  = \sqrt{1 + t^{-2} \lambda_i(B^* B)}\approx 1 + 
\frac{1}{2} t^{-2} \lambda_i(B^* B). 
\end{equation*}
Hence, 
\[
\|E\|_F^2 = \sum_i (\sigma_i(A) - 1)^2 \approx \Bigl(\frac12
t^{-2}\Bigr)^2 \sum_i \lambda_i^2(B^* B) = \Bigl(\frac12 t^{-2}\|B^*
B\|_F\Bigr)^2,
\]
which proves the claim. 
\end{proof}

\subsection{Proof of Corollary \ref{cor:newthm}}
\label{sec:corollaries}

Put $L = f(I_{n-k}, X_2) X_1 \Sigma^{-1}$. Since $f(I_{n-k}, X_2)$
is an orthogonal projector, $ \|f(I_{n-k}, X_2)\| \le 1$ and
\eqref{eq:keytoall} gives
\begin{equation}\label{LandR}
\|L\| \leq W \leq \| L \| + \|f(I_{n-k}, X_2)\| \le \|L\| + 1.  
\end{equation}
% \ejc{Above, we have W <= (\|L\|^2 + 1)^1/2 <= \|L\| + 1/(2 \|L\|).}
Also,
\begin{equation}\label{Lupperbound}
  \|L\| \leq \|f(I_{n-k}, X_2)\| \|X_1\| \|\Sigma^{-1}\|  \leq \|X_1\| \|\Sigma^{-1}\|. 
\end{equation}
Standard estimates in random matrix theory (Lemma
\ref{lem:expectedValue}) give
\begin{equation*}
  \sqrt{n-k}-\sqrt{k} \leq \E  \sigma_{\min} (X_1)  \leq \E  \sigma_{\max}(X_1) \leq \sqrt{n-k} + \sqrt{k}.
\end{equation*}
Hence,
\[
\E \|L\| \le \E \|X_1\| \, \E \|\Sigma^{-1}\| \le
(\sqrt{n-k} + \sqrt{k}) \, \E \|\Sigma^{-1}\|.
\]

Conversely, letting $i$ be the index corresponding to the largest
entry of $\Sigma^{-1}$, we have $\|L\| \ge \|L e_i\|$. Therefore,
\[
\|L\| \ge \|\Sigma^{-1}\| \, \|f(I_{n-k}, X_2) z\|, 
\]
where $z$ is the $i$th column of $X_1$. Now $f(I_{n-k}, X_2) z$ is the
projection of a Gaussian vector onto a plane of dimension $n - k - p$
drawn independently and uniformly at random.  This says that
$\|f(I_{n-k}, X_2) z\| \eqd \sqrt{Y}$, where $Y$ is a chi-square
random variable with $d = n-(k+p)$ degrees of freedom. %  Then if
% $\Gamma(\cdot)$ denotes the $\Gamma$ function, standard calculations
% give
% \[
% \E \|f(I_{n-k}, X_2) z\|_2 = \E \sqrt{Y} =
% \frac{\sqrt{2}\Gamma(\frac{d+1}{2})}{\Gamma(\frac{d}{2})} \approx
% \sqrt{d}.
% \]
% \ejcrw{I use an approximation valid for large $d$ but we can replace
%   this with a concrete bound.} 
If $g$ is a nonnegative random variable, then\footnote{This follows
  from H\"older's inequality $\E |XY| \le (\E |X|^{3/2})^{2/3} (\E
  |Y|^3)^{1/3}$ with $X = g^{2/3}$, $Y = g^{4/3}$.}
\begin{equation}
  \label{eq:holder}
\E g \ge \sqrt{\frac{ (\E g^2)^{3}}{\E g^4}}. 
\end{equation}
Since $\E Y = d$ and $\E Y^2 = d^2 + 2d$, we have
\[
\E \sqrt{Y} \ge \sqrt{\frac{ (\E Y)^{3}}{\E Y^2}} = \sqrt{d} \,
\sqrt{\frac{1}{1 + 2/d}} \ge  \sqrt{d} \sqrt{1 - \frac{2}{d}}. 
\]

Hence,
\[
\E \|L\| \ge \E \|X_1\| \, \E \|\Sigma^{-1}\| \ge
\sqrt{n-(k+p+2)} \, \E \|\Sigma^{-1}\|,
\]
which establishes Corollary \ref{cor:newthm}.

The limit bounds \eqref{eq:infinite-upper} and
\eqref{eq:infinite-lower} are established in a similar manner. The
upper estimate is a consequence of the bound $W \le 1 + \|X_1\|
\|\Sigma^{-1}\|$ together with Lemma \ref{lem:almostSureLaws}. The
lower estimate follows from $W \ge Y^{1/2} \|\Sigma^{-1}\|$, where
$Y$ is a chi-square as before, together with Lemma
\ref{lem:almostSureLaws}. We forgo the details.

\subsection{Proof of Theorem \ref{cor:OUR_MIXEDNORM}}
\label{sec:mixed}

Take $M(t)$ as in \eqref{eq:Sigmat} with $\sigma_{k+1}, \sigma_{k+2},
\ldots, \sigma_n$ on the diagonal of $D_{n-k}$. Applying
\eqref{eq:keytoall} gives 
\begin{equation}
  \label{eq:keytoall2}
\| f(D_{n-k}, X_2) X_1 \Sigma^{-1}\| \le \lim_{t \rightarrow \infty}
\|f(M(t), G)\| \le \| f(D_{n-k}, X_2) X_1 \Sigma^{-1}\| + \|
f(D_{n-k}, X_2)\|.
\end{equation}
We follow \cite[Proof of Theorem 10.6]{HMT} and use an inequality of
Gordon to establish
\[
\E_{X_1} \| f(D_{n-k}, X_2) X_1 \Sigma^{-1}\| \le \| f(D_{n-k}, X_2)\|
\|\Sigma^{-1}\|_F + \| f(D_{n-k}, X_2)\|_F
\|\Sigma^{-1}\|,  
\]
where $\E_{X_1}$ is expectation over $X_1$. Further, it is well-known
that
\[
\E \|\Sigma^{-1}\|_F^2 = \frac{k}{p-1}.  
\]
The reason is that $\|\Sigma^{-1}\|_F^2 =
\operatorname{trace}(M^{-1})$, where $M$ is a Wishart matrix $M \sim
\mathcal{W}(I_k, k+p)$. The identity follows from $\E M^{-1} =
(p-1)^{-1} I_k$ \cite[Exercise 3.4.13]{MardiaBook}. In summary, the
expectation of the right-hand side in \eqref{eq:keytoall2} is bounded
above by
\[
\left(1 + \sqrt{\frac{k}{p-1}}\right) \E \| f(D_{n-k}, X_2)\| + \E
\|\Sigma^{-1}\| \, E \| f(D_{n-k}, X_2)\|_F.
\]
Since $\sigma(f(D_{n-k},X_2)) \le \sigma(D_{n-k})$, the conclusion of
the theorem follows.

% Just as before, 
% \[
%  \| f(D_{n-k}, X_2) X_1
% \Sigma^{-1}\| \geqd \| \Sigma^{-1} \| \, \| f(D,X_2) z\|_2, 
% \]
% where $z$ is an independent Gaussian vector. Therefore, Lemma
% \ref{lem:expected_frob_norm} implies
% \[
% \E \| f(D_{n-k}, X_2) X_1 \Sigma^{-1}\| \ge \frac{1}{\sqrt{3}} \E \|
% \Sigma^{-1} \| \, \E \|f(D,X_2)\|_F.
% \]

\section{Discussion}
\label{sec:discussion}

We have developed a new method for characterizing the performance of a
well-studied algorithm in randomized numerical linear algebra and used
it to prove sharp performance bounds. A natural question to ask when
using Algorithm \ref{alg} is if one should draw $G$ from some
distribution other than the Gaussian. It turns out that for all values
$m,n,k$ and $p$, choosing $G$ with Gaussian entries minimizes
\begin{equation*}
\sup_A \E \| A - QQ^*A\|/\sigma_{k+1}.
\end{equation*}
This is formalized as follows:
\begin{lemma} \label{lem:optimal_gaussian} Choosing $G \in \R^{m
    \times \ell}$ with i.i.d.~Gaussian entries minimizes the supremum
  of $\E \|(I-QQ^*)A \|/\sigma_{k+1}$ across all choices of $A$.
\end{lemma}
\begin{proof}
  Fix $A$ with $\sigma_{k+1}(A) = 1$ (this is no loss of generality)
  and suppose $F$ is sampled from an arbitrary measure with
  probability $1$ of being rank $\ell$ (since being lower rank can
  only increase the error). The expected error is $\E_F
  \|\mathcal{P}_{AF}(A)\|$, where for arbitrary matrices,
  $\mathcal{P}_{A}(B) = (I - P) B$ in which $P$ is the orthogonal
  projection onto the range of $A$. Suppose further that $U$ is drawn
  uniformly at random from the space of orthonormal matrices. Then if
  $G$ is sampled from the Gaussian distribution,
\begin{equation*}
\E_{G} \| \mathcal{P}_{AG} (A) \| = \E_{F,U}\| \mathcal{P}_{AUF} (A) \|
\end{equation*}
since $UF$ chooses a subspace uniformly at random as does
$G$. Therefore, there exists $U_0$ with the property $\E_{F}\|
\mathcal{P}_{A U_0 F} (A) \| \geq \E_G \| \mathcal{P}_{AG} (A)
\|$, whence
\begin{align*}
  \E_G \| \mathcal{P}_{AG} (A) \| \leq \E_{F}\| \mathcal{P}_{ A U_0
    F} (A) \| = \E_{F}\| \mathcal{P}_{ A U_0 F} (A U_0)\|.
\end{align*}
Hence, the expected error using a test matrix drawn from the Gaussian
distribution on $A$ is smaller or equal to that when using a test
matrix drawn from another distribution on $A U_0$. Since the singular
values of $A$ and $A U_0$ are identical since $U_0$ is orthogonal, the
Gaussian measure (or any measure that results in a rotationally
invariant choice of rank $k$ subspaces) is worst-case optimal for the
spectral norm.
\end{proof}

The analysis presented in this paper does not generalize to a test
matrix $G$ drawn from the subsampled random Fourier transform (SRFT)
distribution as suggested in \cite{WLVT_SRFT}. Despite their inferior
performance in the sense of Lemma \ref{lem:optimal_gaussian}, SRFT
test matrices are computationally attractive since they come with fast
algorithms for matrix-matrix multiply.

\appendix

% add "Appendix" to the section heading

\section{Appendix}

% \section{Extremal singular values of Gaussian matrices}

We use well-known bounds to control the expectation of the extremal
singular values of a Gaussian matrix. These bounds are recalled in
\cite{Rudelson2010}, though known earlier. 
\begin{lemma} \label{lem:expectedValue} If $m>n$ and $A$ is a $m
  \times n$ matrix with i.i.d.~$\mathcal{N}(0,1)$ entries, then
\begin{equation*}
\sqrt{m} - \sqrt{n} \leq \E \sigma_{\min}(A)\leq \E \sigma_{\max}(A)  \leq \sqrt{m} + \sqrt{n}.
\end{equation*}
\end{lemma}

Next, we control the expectation of the norm of the pseudo-inverse
$A^\dagger$ of a Gaussian matrix $A$.
\begin{lemma} \label{lem:pseudoInverseExpectedValue}
In the setup of Lemma \ref{lem:expectedValue}, we have 
\begin{equation*}
  \frac{1}{\sqrt{m-n}} \leq \E \|A^\dagger \|  \leq e \frac{\sqrt{m}}{m-n}.
\end{equation*}
\end{lemma}
\begin{proof}
  The upper bound is the same as is used in \cite{HMT} and follows
  from the work of \cite{ChenDongarra}. For the lower bound, set $B =
  (A^*A)^{-1}$ which has an inverse Wishart distribution, and observe that 
\[
\|A^\dagger\|^2 = \|B\| \ge B_{11},
\]
where $B_{11}$ is the entry in the $(1,1)$ position. It is known that
$B_{11}  \eqd 1/Y$, where $Y$ is distributed as a chi-square variable
with $d = m - n + 1$ degrees of freedom \cite[Page 72]{MardiaBook}. Hence, 
\[
\E \|A^\dagger\| \ge \E \frac{1}{\sqrt{Y}} \ge \frac{1}{\sqrt{\E Y}} =
\frac{1}{\sqrt{m-n+1}}.
\]
\end{proof}

The limit laws below are taken from \cite{Silverstein} and
\cite{Geman}. 
\begin{lemma} \label{lem:almostSureLaws} Let $A_{m,n}$ be a sequence
  of $m \times n$ matrix with i.i.d.~$\mathcal{N}(0,1)$ entries such
  that $\lim_{n \rightarrow \infty} m/n = c \ge 1$. Then
\begin{align*}
\frac{1}{\sqrt{n}} \sigma_{\min}(A_{m,n}) & \rightarrowas  \sqrt{c} - 1\\
\frac{1}{\sqrt{n}} \sigma_{\max}(A_{m,n}) & \rightarrowas \sqrt{c}+1.
\end{align*}
\end{lemma}

Finally, the lemma below is used in the proof of Lemma \ref{MANINV}. 
\begin{lemma}\label{lem:chaining}
  Put $f_G(\cdot) = f(\cdot, G)$ for convenience. Take $A \in \R^{m
    \times n}$ and $G = [G_1, G_2, \ldots, G_k]$, with each $G_i \in
  \R^{n \times \ell_i}$. Then
 \[
 	(f_{G_k} \circ f_{G_{k-1}} \ldots \circ f_{G_1})(A) = f_G(A).
\]
This implies that $ A_{{(j)}_{(p)}} = A_{(j+p)}$ and $A_{(1)_{\ldots_{(1)}}}$
  ($\ell$ times) $= A_{(\ell)}$.
\end{lemma}
\begin{proof}
  We assume $k=2$ and use induction for larger $k$. $f_{G_1}(A)$ is
  the projection of $A$ onto the subspace orthogonal to $\spn(AG_1)$
  and $f_{G_2} \circ f_{G_1}(A)$ is the projection of $f_{G_1}(A)$
  onto the subspace orthogonal to $\spn( f_{G_1}(A)G_2)$. However,
  $\spn(AG_1)$ and $\spn( f_{G_1}(A)G_2)$ are orthogonal subspaces and,
  therefore, $f_{G_2} \circ f_{G_1}(A)$ is the projection of $A$ onto
  the subspace orthogonal to $\spn( AG_1) \oplus \spn(f_{G_1}(A)G_2) =
  \spn(AG_1) \oplus \spn(AG_2)= \spn (AG)$; this is $f_G(A)$.
\end{proof}

% \begin{lemma}\label{lem:expected_frob_norm}
%   Let $A$ be a fixed matrix with singular values denoted by
%   $\sigma_i$, and $z = (z_1, \ldots, z_n)$ be a Gaussian vector with
%   i.i.d.~$\mathcal{N}(0,1)$ entries. Then
% \[
% \E \|A z\|_2 \geq  \sqrt{\frac{1}{1 + 2a}} \, \|A\|_F \ge
% \frac{1}{\sqrt{3}} \, \|A\|_F, \qquad a := \frac{\sum_i
%   \sigma_i^4}{\left(\sum_i \sigma_i^2\right)^2}.
%  \]
% \end{lemma}
% \begin{proof}
%   Observe that the second inequality is a consequence of $a \le 1$.
%   Now by rotational invariance, we can assume that $A$ is diagonal
%   with diagonal entries $\sigma_i$, whence
% \[
%  \E \|A z\|_2 = \E \sqrt{\sum_i \sigma_i^2 z_i^2}. 
% \]
% Set $g = \|Az\|_2$. Then $\E g^2 = \sum_i \sigma_i^2$ and we compute 
% \[
% \E g^4 = \bigl(\sum_i \sigma_i^2\bigr)^2 + 2 \sum_i \sigma_i^4. 
% \]
% The conclusion follows from \eqref{eq:holder}. 
% \end{proof}

\small 

\subsection*{Acknowledgements}

E. C. is partially supported by NSF via grant CCF-0963835 and by a
gift from the Broadcom Foundation. We thank Carlos Sing-Long for
useful feedback about an earlier version of the manuscript.  These
results were presented in July 2013 at the European Meeting of
Statisticians.

  \end{document}